\documentclass[3p]{elsarticle}
\usepackage{amssymb}
\usepackage{amsfonts}
\usepackage{tikz}
\usetikzlibrary{decorations.pathmorphing,patterns,arrows,shapes,snakes,automata,backgrounds,petri,calc}
\newtheorem{theorem}{Theorem}
\newtheorem{conjecture}[theorem]{Conjecture}

\newtheorem{lemma}[theorem]{Lemma}
\newtheorem{observation}[theorem]{Observation}

\newtheorem{claim}{Claim}

\newproof{proof}{Proof}

\begin{document}
\begin{frontmatter}

\title{On the neighbour sum distinguishing index of graphs with
bounded maximum average degree}

\author[LaBRI]{H. Hocquard \fnref{FRgrant}}
\ead{herve.hocquard@labri.fr}

\author[agh]{J. Przyby{\l}o\corref{cor1}\fnref{grantJP,MNiSW}} 
\ead{jakubprz@agh.edu.pl}

\cortext[cor1]{Corresponding author}
\fntext[FRgrant]{Supported by CNRS-PICS Project no. 6367 ``GraphPar''.}
\fntext[grantJP]{Supported by the National Science Centre, Poland, grant no. 2014/13/B/ST1/01855.}
\fntext[MNiSW]{Partly supported by the Polish Ministry of Science and Higher Education.}

\address[LaBRI]{LaBRI (Universit\'e de Bordeaux), 351 cours de la Lib\'eration, 33405 Talence Cedex, France}
\address[agh]{AGH University of Science and Technology, al. A. Mickiewicza 30, 30-059 Krakow, Poland}

\begin{abstract}
A proper edge $k$-colouring of a graph $G=(V,E)$ is an assignment $c:E\to \{1,2,\ldots,k\}$ of colours to the edges of the graph such that no two adjacent edges are associated with the same colour. A neighbour sum distinguishing edge $k$-colouring, or nsd $k$-colouring for short, is a proper edge $k$-colouring such that $\sum_{e\ni u}c(e)\neq \sum_{e\ni v}c(e)$ for every edge $uv$ of $G$. We denote by $\chi'_{\sum}(G)$ the neighbour sum distinguishing index of $G$, which is the least integer $k$ such that an nsd 
$k$-colouring of $G$ exists.
By definition 
at least maximum degree, $\Delta(G)$ colours are needed for this goal.
In this paper we prove that $\chi'_\Sigma(G) \leq \Delta(G)+1$ for any graph $G$ without isolated edges, and with
${\rm mad}(G)<3$, $\Delta(G) \geq 6$.

\end{abstract}


\end{frontmatter}

\section{Introduction}
A \emph{proper edge $k$-colouring} of a graph $G=(V,E)$ is an assignment 
of colours to the edges of the graph such that two adjacent edges do not host the same colour.
We use the standard notation, $\chi'(G)$, to denote the chromatic index of $G$.
A \emph{neighbour sum distinguishing} edge $k$-colouring, or \emph{nsd} $k$-colouring for short,
is a proper edge colouring $c:E\to \{1,2,\ldots,k\}$ such that for every edge $uv\in E$, there is no \emph{conflict} between $u$ and $v$,
{\it i.e.}, $s(u)\neq s(v)$, where $s(u)$ is the sum of colours taken on the edges incident with $u$.
In other words, for every vertex $u\in V$, $\displaystyle s(u) = \sum_{e \in E_u} c(e)$, where $E_u$ is the set of edges incident with $u$ in $G$.
We denote by $\chi'_{\sum}(G)$ the \emph{neighbour sum distinguishing index} of $G$, which is the least integer $k$ such that an nsd (edge) $k$-colouring of $G$ exists.
This graph invariant binds its two famous archetypes - the parameter associated with so called Zhang's Conjecture~\cite{Zhang},
where the required distinction is weaker and concerns sets of colours rather than their sums,
cf.~\cite{BalGLS,BonamyEtAl,Hatami,HocqMont,Hornak_planar,WangWang,Zhang} for representative results concerning it,
and the problem commonly referred to as 1--2--3 Conjecture~\cite{123KLT},
whose objective 
were not necessarily proper edge colourings in turn,
see~\cite{Louigi30,Louigi,KalKarPf_123,WangYu} for a few 
breakthroughs concerning this. 
The roots of this branch of graph theory date back to the 1980s, and the papers~\cite{ChartrandErdosOellermann,Chartrand}
on degree irregularities in graphs (and multigraphs) and the parameter \emph{irregularity strength} of a graph.
There first
integer edge weights (colours) became of use to represent the multiplicities of respective
edges in an investigated multigraph with a given underlying simple graph.
The sum $s(u)$ defined above corresponds then to the degree of a given vertex $u$ in the multigraph with underlying graph $G$,
see~\cite{Aigner,
Lazebnik,
Frieze,KalKarPf,Lehel,MajerskiPrzybylo2,Nierhoff} 
for more details and a few crucial results on the irregularity strength.

Note that as for other graph invariants of this type, the value of $\chi'_{\sum}(G)$ is well defined for all graphs without isolated edges.
By definition, the neighbour sum distinguishing index of every such graph $G$ is not smaller than $\chi'(G)$,
while by Vizing's theorem, $\chi'(G)$ equals the maximum degree of $G$, $\Delta(G)$, or $\Delta(G)+1$.
The following conjecture was proposed by Flandrin {\it et al.} in~\cite{FlandrinMPSW},
where it was also verified for a few classical graph families, including, {\it e.g.},
paths, cycles, complete graphs, complete bipartite graphs and trees.

\begin{conjecture}\label{Flandrin_et_al_Conjecture}
If $G$ is a connected graph of order
at least three different from the cycle $C_5$, then
$\chi'_{\sum}(G) \leq \Delta(G) + 2$.
\end{conjecture}

In general it is known that this conjecture is asymptotically correct, as confirmed by
the following probabilistic result of Przyby{\l}o from~\cite{Przybylo_asym_optim}.

\begin{theorem}
\label{Th_sum_asymptotic}
If $G$ is a connected graph of maximum degree $\Delta\geq 2$, then $\chi'_{\sum}(G)\leq (1+o(1))\Delta$.
\end{theorem}

Other upper bounds can be found in~\cite{FlandrinMPSW,Przybylo_CN_1,Przybylo_CN_2,WangYan_sum}.
Recently, Bonamy and Przyby{\l}o \cite{BP14} also confirmed Conjecture~\ref{Flandrin_et_al_Conjecture}
for planar graphs with sufficiently large maximum degree 
proving that:

\begin{theorem} 
Any planar graph $G$ with $\Delta(G) \geq 28$ and with no isolated edges satisfies $\chi'_\Sigma(G) \leq \Delta(G)+1$.
\end{theorem}

Let ${\rm mad}(G)=\max\left\{\frac{2|E(H)|}{|V(H)|},\;H \subseteq G\right\}$ be the maximum average degree of the graph $G$, where $V(H)$ and $E(H)$ are the sets of vertices and edges of $H$, respectively.
This is a conventional measure of sparseness of an arbitrary graph (not necessary planar). For more details on this invariant see \cite{Toft}, where properties of the maximum average degree are exhibited and where it is proved that maximum average degree may be computed by a polynomial algorithm. Moreover it can be efficiently computed by translating the question into a flow problem on the right graph \cite{Coh10}.

Dong {\it et al.} first made the link between maximum average degree and neighbour sum distinguishing index \cite{DongWang_mad}. They proved the following result.

\begin{theorem} \label{th:summadCite1} 
Any graph $G$ with no isolated edges, $\Delta(G) \geq 6$ and ${\rm mad}(G)<\frac52$ satisfies $\chi'_\Sigma(G) \leq \Delta(G)+1$.
\end{theorem}

This subject was intensively studied afterwards, and the following improvements have been provided.

\begin{theorem} \cite{HuChenLuoMiao} \label{th:summadCite2}
Any graph $G$ with no isolated edges, $\Delta(G) \geq 6$ and ${\rm mad}(G)< \frac{8}{3}$ satisfies $\chi'_\Sigma(G) \leq \Delta(G)+1$.
\end{theorem}

\begin{theorem} \cite{HuChenLuoMiao,YuQuWangWang} \label{th:summadCite3}
Any graph $G$ with no isolated edges, $\Delta(G) \geq 5$ and ${\rm mad}(G)< 3$ satisfies $\chi'_\Sigma(G) \leq \Delta(G)+2$.
\end{theorem}


In this paper, we strengthen all three results above by proving the following (Note that in fact Theorem~\ref{th:summad} below
implies all Theorems~\ref{th:summadCite1}--\ref{th:summadCite3} above).

\begin{theorem}\label{th:summad}
Any graph $G$ with no isolated edges, $\Delta(G) \geq 6$ and ${\rm mad}(G)<3$ satisfies $\chi'_\Sigma(G) \leq \Delta(G)+1$.
\end{theorem}


\section{Proof of Theorem~\ref{th:summad}}

\subsection{Preliminaries}

Fix an integer $k\geq 6$.
In the following, $n_i(G)$ denotes the number of vertices of degree $i$ in a graph $G$. 
We say a graph $G$ is \emph{smaller} than a graph $H$, $G \prec H$ if 
$(n_k(G),\ldots,n_2(G),n_1(G))$ precedes $(n_k(H),\ldots,n_2(H),n_1(H))$ with respect to the standard lexicographic order.
We say a graph is \emph{minimal} for a property when no smaller graph verifies it.
We shall also call any vertex of degree $d$ ($\geq d$, $\leq d$) in a given graph a \emph{$d$-vertex} (\emph{$d^+$-vertex}, \emph{$d^-$-vertex}, resp.) of this graph.
The same nomenclature shall be used for neighbours as well.\\



\subsection{Structural properties of H}


Suppose $H$ is a minimal graph without isolated edges such that $\Delta(H) \leq k$, and ${\rm mad}(H)<3$ 
and $\chi'_\Sigma(H)>k+1$. In the remaining part of the paper we argument that in fact $H$ cannot exist, and thus prove Theorem~\ref{th:summad}.

In this subsection, we exhibit some structural properties of $H$.
%
%
The following lemma shall be very useful to this end.
Its proof is inspired by the research from~\cite{BP14}. 

\begin{lemma}\label{MartheLemma}
For any finite sets $L_1,\ldots,L_t$ of real numbers with $|L_i|\geq t$ for $i=1,\ldots,t$,
the set $\{x_1+\ldots+x_t:x_1\in L_1,\ldots,x_t\in L_t; x_i\neq x_j~{\rm for}~i\neq j\}$ contains
at least $\sum_{i=1}^t|L_i|-t^2+1$ distinct elements.
\end{lemma}

\begin{proof}
We begin by first dynamically modifying the lists $L_1,\ldots,L_t$.
Thus subsequently, for $i=1,2,\ldots,t-1$, we take $\min L_i$
(where every $L_p$, $p\in\{1,\ldots,t\}$ shall always refer to the up-to-date remainder of this list on a given stage of our modifying procedure) and remove
it from all current lists $L_j$ with $j>i$.
Then, subsequently, for $i=t,t-1,\ldots,2$, we find $\max L_i$ and remove
it from all up-to-date lists $L_j$ with $j<i$,
and denote the finally constructed respective lists by $L'_1,\ldots,L'_t$.
As a result at most $t-1$ elements were removed from every list $L_i$
and
for every $i<j<l$, $L'_j$ contains neither $\min L'_i$
nor $\max L'_l$.
Let $L'_i=\{c_{i,1},c_{i,2},\ldots,c_{i,l_i}\}$, where $c_{i,1}<c_{i,2}<\ldots<c_{i,l_i}$, for $i=1,\ldots,t$.
Then it is straightforward to see that the following $\sum_{i=1}^t|L'_i|-t+1\geq \sum_{i=1}^t(|L_i|-(t-1))-t+1 =\sum_{i=1}^t|L_i|-t^2+1$ sums are distinct and
each consists of $t$ pairwise distinct integers:
\begin{eqnarray*}
&&c_{1,1}+c_{2,1}+\ldots+c_{t-2,1}+c_{t-1,1}+c_{t,1}\\
&<&c_{1,1}+c_{2,1}+\ldots+c_{t-2,1}+c_{t-1,1}+c_{t,2}\\
&&\ldots\\
&<&c_{1,1}+c_{2,1}+\ldots+c_{t-2,1}+c_{t-1,1}+c_{t,l_{t}}\\
&<&c_{1,1}+c_{2,1}+\ldots+c_{t-2,1}+c_{t-1,2}+c_{t,l_{t}}\\
&&\ldots\\
&<&c_{1,l_1}+c_{2,l_2}+\ldots+c_{t-2,l_{t-2}}+c_{t-1,l_{t-1}}+c_{t,l_{t}}.
\end{eqnarray*}
$\blacksquare$
\end{proof}


A $2$-vertex or a $3$-vertex is called \emph{bad} if it is adjacent to a vertex of degree $2$.
Otherwise these are called \emph{good}. A vertex is called \emph{deficient} if it is a $1$-vertex or a bad $2$-vertex,
while a vertex is referred to as \emph{half-deficient} if it is a good $2$-vertex or a bad $3$-vertex.


\begin{claim}\label{claimstructure}
The graph $H$ does not contain any of:

\begin{enumerate}
\item[(C1)] \label{c1} a $1$-vertex $v$ adjacent to a $(\frac{k}{2}+1)^-$-vertex $u$;
\item[(C2)] \label{c2} a $2$-vertex $v$ adjacent to a $(\frac{k+1}{2})^-$-vertex $u$ and to a $(\frac{k}{2})^-$-vertex $w$, $u\neq w$;
\item[(C3)] \label{c3} a $3$-vertex $v$ adjacent to a $(\frac{k}{2})^-$-vertex $u$ and to a $2$-vertex $w$, $u\neq w$;
\item[(C4)] \label{c4} a triangle $uvw$ with $d(u)=2=d(w)$;
\item[(C5)] \label{c5} a vertex $v$ adjacent to a $1$-vertex $u$ and to a bad $2$-vertex $w$;
\item[(C6)] \label{c6} a vertex $v$ adjacent with two bad $2$-vertices $u$ and $w$;
\item[(C7)] \label{c7} a vertex $v$ adjacent with two $1$-vertices $u_1,u_2$ and to a half-deficient vertex $w$;
\item[(C8)] \label{c8} a vertex $v$ of degree $d\geq 3$ adjacent to $d-2$ vertices $u_1,\ldots,u_{d-2}$ of degree $1$;
\item[(C9)] \label{c9} a vertex $v$ of degree $d\leq \frac{2}{3}k$ 
adjacent with a bad $2$-vertex $u$ and to a half-deficient vertex $w$;
\item[(C10)] \label{c10} a vertex $v$ of degree $d$ adjacent to exactly one bad $2$-vertex $u$, at least one half-deficient vertex and to at most $k-d$ vertices which are neither deficient nor half-deficient;
\item[(C11)] \label{c11} a vertex $v$ of degree $d$ adjacent to exactly one $1$-vertex $u$
and to at most $k-d+1$ vertices which are neither deficient nor half-deficient;
\item[(C12)] \label{c12} a $5$-vertex $v$ adjacent to $5$ half-deficient vertices $u_1,...,u_5$;
\item[(C13)] \label{c13} a $4$-vertex $v$ adjacent to at least $3$ half-deficient vertices $u_1,u_2,u_3$.
\end{enumerate}
\end{claim}
\medskip


\begin{figure}[ht]
\begin{center}
\centering
\begin{tikzpicture}[scale=0.8,auto]
	\tikzstyle{w}=[draw,circle,fill=white,minimum size=5pt,inner sep=0pt]
	\tikzstyle{b}=[draw,circle,fill=black,minimum size=5pt,inner sep=0pt]
	\tikzstyle{t}=[rectangle,minimum size=5pt,inner sep=0pt]

	 \tikzstyle{whitenode}=[draw,ellipse,fill=white,minimum size=9pt,inner sep=0pt]
	 \tikzstyle{blacknode}=[draw,circle,fill=black,minimum size=9pt,inner sep=0pt]
	 \tikzstyle{texte}=[minimum size=9pt,inner sep=0pt]


\draw (0,0) node[b] (v) [label=left:$v$] {} 
-- ++(0:3cm) node[whitenode] (u) [label=above:$u$] {$\left(\frac{k}{2}+1\right)^-$};




\draw (1.7,-1.5) node[t] (t1) {$(C_1)$};

\draw (6,0) node[whitenode] (b) [label=above:$u$] {$\left(\frac{k+1}{2}\right)^-$}
--++ (0:1.75cm) node[b] (v1) [label=above:$v$] {}
--++ (0:1.5cm) node[whitenode] (b) [label=above:$w$] {$\left(\frac{k}{2}\right)^-$};

\draw (7.8,-1.5) node[t] (t1) {$(C_2)$};



\draw (11,0) node[w] (u) {}
--++ (0:1.2) node[b] (v1) [label=below:$w$] {}
--++ (0:1.2) node[b] (v2) [label=below:$v$] {}
--++ (0:1.2) node[whitenode] (v3) [label=above:$u$] {$\left(\frac{k}{2}\right)^-$};

\draw (13.4,1.2) node[w] (v4) [label=above:$v_1$] {};


\draw (v2) -- (v4);

\draw (13,-1.5) node[t] (t1) {$(C_3)$};

\draw (16.5,0) node[b] (u) [label=below:$u$] {}
--++ (60:1.5cm) node[w] (v1) [label=above:$v$] {}
--++ (-60:1.5cm) node[b] (v2) [label=below:$w$] {};

\draw (v2) -- (u);

\draw (17.3,-1.5) node[t] (t1) {$(C_4)$};


\draw (0,-3.5) node[w] [label=left:$v$] (u) {}
--++ (30:1cm) node[b] [label=above:$w$] (v1) {}
--++ (0:1cm) node[b] (v2) [label=above:$w'$] {}
--++ (0:1cm) node[w] (v3) {};

\draw (u)
--++ (-30:1cm) node[b] [label=right:$u$] (v4) {};

\draw (1.7,-5) node[t] (t1) {$(C_5)$};

\draw (6.1,-3.5) node[w] (u) [label=left:$v$] {}
--++ (30:1cm) node[b] (v1) [label=above:$u$] {}
--++ (0:1cm) node[b] (v2) [label=above:$u'$] {}
--++ (0:1cm) node[w] (v3) {};

\draw (u)
--++ (-30:1cm) node[b] (v4) [label=above:$w$] {}
--++ (0:1cm) node[b] (v5)  [label=above:$w'$] {}
--++ (0:1cm) node[w] (v6) {};

\draw (7.8,-5) node[t] (t1) {$(C_6)$};

\draw (11.3,-3.5) node[w] (u) [label=left:$v$] {}
--++ (30:1.5cm) node[w] (v1) [label=above:$w$] {\large $h$}
--++ (0:1.5cm) node[w] (v2) [label=above:$w'$] {\small $3^+$};

\draw (u)
--++ (0:1.3cm) node[b] (v3) [label=right:$u_1$] {};

\draw (u)
--++ (-30:1.5cm) node[b] [label=right:$u_2$] (v4) {};

\draw (13,-5) node[t] (t1) {$(C_7)$};

\draw (16.8,-3.5) node[w] (u) [label=left:$v$] {\small $3^+$};

\draw (u)
--++ (30:1.5cm) node[b] [label=right:$u_1$] (v1) {};

\draw (u)
--++ (-30:1.5cm) node[b] [label=right:$u_{d-2}$] (v3) {};


\draw (v1) edge [bend left,loosely dotted,thick] node {} (v3);




\draw (17.3,-5) node[t] (t1) {$(C_8)$};


\draw (0,-7) node[whitenode] (u) [label=left:$v$] {\small $\left(\frac{2k}{3}\right)^-$}
--++ (30:1.5cm) node[w] (v1)  [label=above:$w$] {\large $h$}
--++ (0:1cm) node[w] (v2)  [label=above:$w'$] {\small $3^+$}
--++ (0:1cm) node[w] (v3) {};

\draw (u)
--++ (-30:1.5cm) node[b] (v4)  [label=above:$u$] {}
--++ (0:1cm) node[b] (v5)  [label=above:$u'$] {}
--++ (0:1cm) node[w] (v6) {};

\draw (1.7,-8.5) node[t] (t1) {$(C_9)$};

\draw (8.1,-7) node[whitenode] (u) [label=above :$v$] {\small $d$}
--++ (130:1.5cm) node[w] (v1) {\large $h$};

\draw (u)
--++ (0:-1.1cm) node[b] (v2)  [label=above:$u$] {}
--++ (0:-1.1cm) node[b] (v3)  [label=above:$u'$] {}
--++ (0:-1.1cm) node[w] (v4) {};

\draw (u)
--++ (30:1.5cm) node[w] [label=right:$u_1$] (v5) {};

\draw (u)
--++ (-30:1.5cm) node[w] [label=right:$u_{k-d}$] (v6)  {};

\draw (v5) edge [bend left,loosely dotted,thick] node {} (v6);

\draw (7.8,-8.5) node[t] (t1) {$(C_{10})$};

\draw (12.8,-7) node[whitenode] (u) [label=above :$v$] {\small $d$}
--++ (0:-1.3cm) node[b] (v1)  [label=above:$u$] {};

\draw (u)
--++ (30:1.5cm) node[w] (v2) [label=right:$u_1$] {};

\draw (u)
--++ (-30:1.5cm) node[w] (v3) [label=right:$u_{k-d+1}$] {};

\draw (v2) edge [bend left,loosely dotted,thick] node {} (v3);

\draw (13,-8.5) node[t] (t1) {$(C_{11})$};

\draw (16.8,-7) node[b] (u) [label=left:$v$] {};

\draw (u)
--++ (30:1.5cm) node[w] [label=right:$u_1$] (v1) {\large $h$};

\draw (u)
--++ (-30:1.5cm) node[w] [label=right:$u_5$] (v3)  {\large $h$};


\draw (v1) edge [bend left,loosely dotted,thick] node {} (v3);




\draw (17.3,-8.5) node[t] (t1) {$(C_{12})$};


\draw (1.5,-10.5) node[b] [label=above:$v$] (u) {}
--++ (30:1.5cm) node[w] [label=right:$u_1$] (v1) {\large $h$};

\draw(u)
--++ (0:1.3cm) node[w] [label=right:$u_2$] (v2) {\large $h$};

\draw(u)
--++ (-30:1.5cm) node[w] [label=right:$u_3$] (v3) {\large $h$};

\draw(u)
--++ (0:-1.5cm) node[w] (v4) {};

\draw (1.7,-12) node[t] (t1) {$(C_{13})$};

\end{tikzpicture}
\caption{Forbidden configurations in $H$ (where solid vertices have degrees as presented in the figure, hollow vertices may have additional edges and may coincide with other vertices, while the label `h' indicates a half deficient vertex).}\label{FCG_fig}
\end{center}
\end{figure}
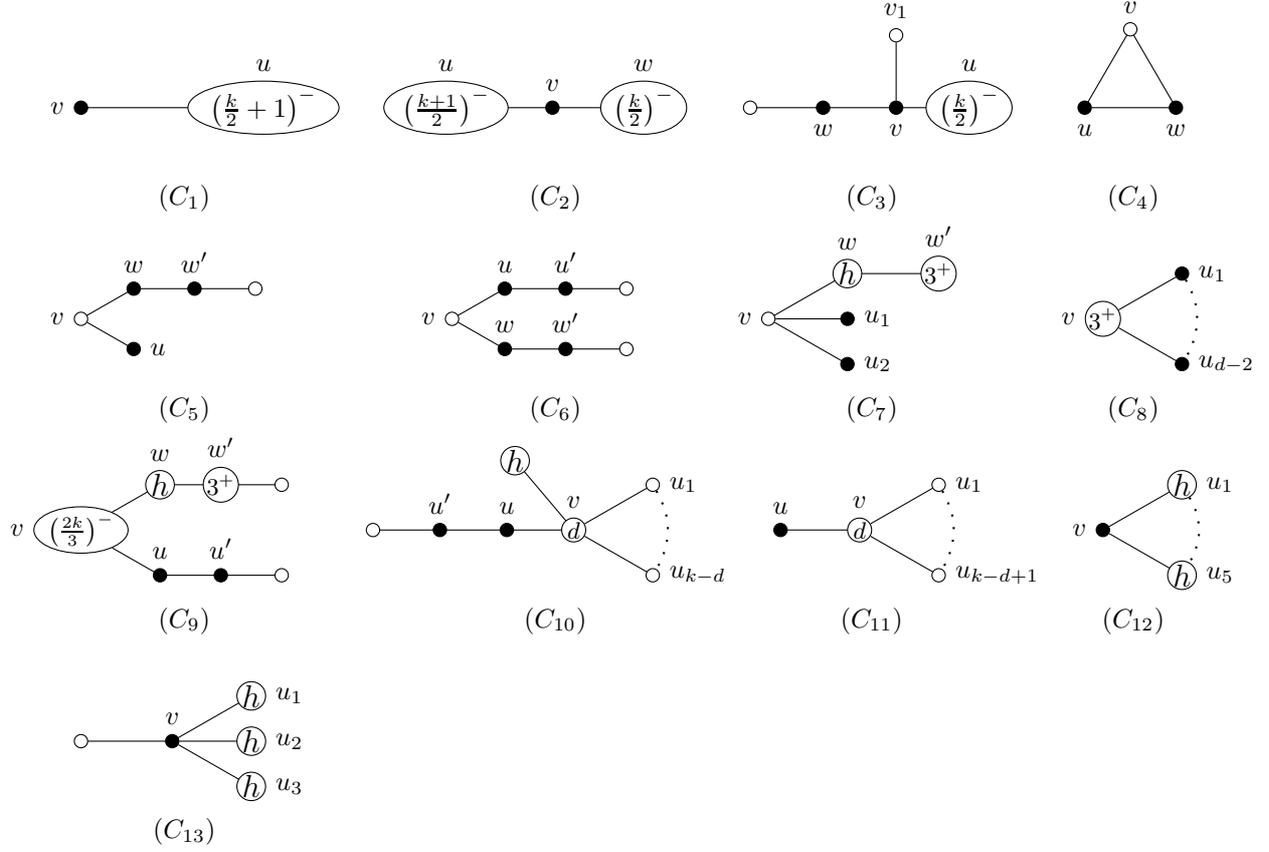

\begin{proof}
We shall argument `reducibility' of each of these 13 configurations separately,
following a similar pattern of reasoning.
I.e., we shall first suppose by contradiction that a given configuration exists in $H$.
Then we shall consider a graph $H'$ smaller than $H$ with $\Delta(H')\leq k$ and $\rm{mad}(H')<3$
(usually guaranteing these properties by constructing $H'$ simply via deleting some edges or vertices from $H$),
and \emph{colour} it \emph{by minimality},
what shall mean from now on that we choose any nsd $(k+1)$-colouring for every component of $H'$ of order at least $3$
(such colouring exists as this component is obviously smaller than $H$ then; cf. the definition of $H$) and fix arbitrarily a colour in $\{1,2,\ldots,k+1\}$ for every isolated edge of $H'$.
%
Finally, in each case, we shall obtain a contradiction by extending the colouring chosen to an nsd $(k+1)$-coloring of the entire $H$.\\

First note that a vertex of degree $d$ shall certainly be sum-distinguished from its $2$-neighbour if
\begin{equation}\label{distinguished2vertices_ineq}
d > \frac{1}{2}(\sqrt{8k+9}+1). 
\end{equation}
Indeed, this inequality is equivalent to $1+2+\ldots+{d-1}>k+1$
(while a colour of an edge joining two vertices is counted in the sums of the both vertices).
Note that this holds {\it e.g.} for 
$d\geq \frac{k+3}{2}$ and for $d\geq\frac{2k+1}{3}$, as $k\geq 6$.
Obviously, a $1$-vertex is always sum-distinguished from its neighbour in $H$.

\begin{enumerate}
\item Suppose there exists a $1$-vertex $v$ adjacent to a $(\frac{k}{2}+1)^-$-vertex $u$. Colour $H'=H-v$ by minimality.
     In order to colour $uv$ then so that the $(k+1)$-colouring of $H$ obtained is proper we have to avoid at most $\frac{k}{2}$ colours,
     and possibly at most $\frac{k}2$ more colours to ensure the sum-distinction (of $u$ from its neighbours other than $v$). Hence, we have at least one colour left to extend the colouring
     to an nsd $(k+1)$-colouring of $H$, a contradiction.\\

\item Assume there exists a $2$-vertex $v$ adjacent to a $(\frac{k+1}{2})^-$-vertex $u$ and to a $(\frac{k}{2})^-$-vertex $w$, $u\neq w$.
     Colour $H'=H-v$ by minimality. Then we colour $vu$ first so that the (partial) $(k+1)$-colouring obtained is proper (at most $\frac{k+1}{2}-1=\frac{k-1}{2}$ forbidden colours), $s(v)\neq s(w)$ ($1$ constraint), and $u$ is sum-distinguished from its (at most $\frac{k-1}{2}$) neighbours with fixed sums, hence we have at least one colour available for this aim.
Finally colour $vw$ so that the colouring is proper (at most $(\frac{k}{2}-1)+1=\frac{k}{2}$ constraints) and $v$ and $w$ are sum-distinguished from their neighbours other than $v$ and $w$ (again $\frac{k}{2}$ constraints). Hence, we obtain a contradiction, as we have at least one colour left to extend the colouring.\\

\item Suppose there exists a $3$-vertex $v$ adjacent to a $(\frac{k}{2})^-$-vertex $u$ and to a $2$-vertex $w$, $u\neq w$.
     Denote by $v_1$ the third neighbour of $v$ distinct from $u$ and $w$. Colour $H'=H-\{uv,vw\}$ by minimality.
     Then colour $uv$ so that the colouring is proper (at most $(\frac{k}{2}-1)+1=\frac{k}{2}$ constraints), $s(v)\neq s(w)$ ($1$ constraint), and $u$ is sum-distinguished from its (at most $\frac{k}{2}-1$) neighbours with fixed sums. Finally colour $vw$ so that the colouring is proper and $v$ and $w$ are sum-distinguished from their neighbours other than $v$ and $w$. We can extend the colouring, since we have at least $7$ colours available (where $7 \geq 2\times 3=6$), a contradiction.\\

\item Assume there exists a triangle $uvw$ with $d(u)=2=d(w)$. Colour $H'=H-uw$ by minimality. Note that $s(u)\neq s(w)$ then, as $uv$ and $wv$ must be coloured differently. Then colour $uw$ so that the colouring is proper ($2$ constraints), $s(u)\neq s(v)$ ($1$ constraint) and $s(w)\neq s(v)$ ($1$ constraint). We can extend the colouring, since we have more than $4$ colours available, a contradiction.\\

\item Suppose there exists a vertex $v$ adjacent to a $1$-vertex $u$ and to a bad $2$-vertex $w$. Denote by $w'$ the neighbour of degree $2$ of $w$ ($w'\neq v$). Colour $H'=H-ww'$ by minimality. Next switch the colours of $vu$ and $vw$ if necessary so that $s(w)\neq s(w')$. Then we easily choose a colour for $ww'$ so that the colouring obtained is proper and no sum conflict arises. Hence, the colouring is extended, a contradiction.\\

\item Assume there exists a vertex $v$ adjacent with two bad $2$-vertices $u$ and $w$. Let $u'$ (resp. $w'$) be the neighbour of degree $2$ of $u$ (resp. $w$), $u',w'\neq v$. By $(C2)$, $d(v)\geq 4$ and $u'\neq w'$. Denote by $u''$ (resp. $w''$) the second neighbour of $u'$ (resp. $w'$) distinct from $u$ (resp. $w$). By $(C4)$, $w'\neq u$ and $u'\neq w$, while by $(C2)$, $u'$ and $w'$ cannot be adjacent in $H$. Consider $H'=H+u'w'-\{uu',ww'\}$,
and note that $H'\prec H$, as we have decreased the number of vertices of degree $2$, creating no new vertices of larger degrees at the same time.
 It also holds that ${\rm mad}(H')<3$ (as otherwise there would have to exist a subgraph $H''$ of $H$ with ${\rm mad}(H'')\geq 3$, \emph{e.g.}, $H''=H'-\{u',w'\}$, a contradiction).
 Consequently, we may colour $H'$ by minimality. Hence $vu$ and $vw$ are coloured differently, and the same holds for $u'u''$ and $w'w''$ (if they were coloured the same, there would be a conflict between $u'$ and $w'$ in $H'$). Then we switch the colours of $vu$ and $vw$ if necessary, so that $vu$ (resp. $vw$) and $u'u''$ (resp. $w'w''$) are coloured differently, in order to ensure the sum-distinction between $u$, $u'$ and $w$, $w'$ in $H$ (where the edge $u'w'$ is not taken into account anymore, as it does not appear in $H$). It then suffices to colour the edges $uu'$ and $ww'$ with colours different from these of their respective adjacent edges and such that $s(u),s(w) \neq s(v)$, $s(u')\neq s(u'')$ and $s(w')\neq s(w'')$. This is possible as there are at least $3$ available colours left for this aim in both cases, a contradiction.\\

\item Suppose there is a vertex $v$ adjacent with two $1$-vertices $u_1,u_2$ and to a half-deficient vertex $w$. Let $w'$ be the neighbour of $w$ of degree greater than $2$ other than $v$ (cf. $(C1)$ and $(C3)$). We consider two cases:
\begin{itemize}
\item First suppose $w$ is a good $2$-vertex. We create $H'$ of $H$ by splitting the vertex $w$ in two $1$-vertices $w_1$ and $w_2$ such that $w_1$ is adjacent to $v$ and $w_2$ is adjacent to $w'$. Obviously, $H'\prec H$ 
    and ${\rm mad}(H') \le {\rm mad}(H)$. Hence, we may colour $H'$ by minimality.
    Then we switch the colour of $vw_1$ with the colour of $vu_1$ or $vu_2$ if necessary so that the colour of $vw_1$ is distinct from the colour of $w_2w'$ and $s(w)\neq s(w')$ after identifying back $w_1$ with $w_2$. Since by $(C1)$, $d(v)\geq \frac{k+3}{2}$, then by (\ref{distinguished2vertices_ineq}), $s(v)\neq s(w)$, hence we obtain an nsd $(k+1)$-colouring of $H$, a contradiction.

\item Assume now that $w$ is a bad $3$-vertex. Let $w''$ be the third neighbour of $w$ ({\it i.e.}, $w''\neq v$, $w''\neq w'$ and $d(w'')=2$). We split $w$ into a $1$-vertex $w_1$ adjacent to $v$ and a $2$-vertex $w_2$ adjacent to $w'$ and $w''$. One can observe that the obtained new graph $H'$ is smaller than $H$ (because one $3$-vertex has been removed) and ${\rm mad}(H') \le {\rm mad}(H)$. Hence, we may colour $H'$ by minimality. Then we switch the colour of $vw_1$ with the colour of $vu_1$ or $vu_2$ if necessary so that the colour of $vw_1$ is distinct from the colour of $w_2w'$ and $s(w)\neq s(w'')$ after identifying back $w_1$ with $w_2$. If there are still some colour or sum conflicts in $H$, we change the colour of $ww''$ to eliminate all of these. This is feasible as we have more than $6$ colours available.
    Hence, we can extend the colouring, a contradiction.\\

\end{itemize}



\item Assume there is a vertex $v$ of degree $d\geq 3$ adjacent to $d-2$ vertices $u_1,\ldots,u_{d-2}$ of degree $1$.
By $(C1)$, $d\geq 5$. Colour $H'=H-\{u_1,\ldots,u_{d-2}\}$ by minimality. Then every edge $vu_i$ for $i \in \{1,\ldots,d-2\}$ has $2$ forbidden colours, {\it i.e.}, $(k+1)-2=k-1$ available colours left. By Lemma~\ref{MartheLemma}, we may complete the proper colouring of $H$ in different ways, obtaining at least $(d-2)(k-1)-(d-2)^2+1=(k-d+1)(d-2)+1\geq (d-2)+1\geq 4$ distinct sums for $v$. Since $v$ has at most two neighbours of degree greater than $1$, then, at least one of these $4$ sums is distinct from the sums of these at most two neighbours. Thus again we can extend the colouring, a contradiction.\\

\item Suppose there is a vertex $v$ of degree $d\leq \frac{2}{3}k$ adjacent to a bad $2$-vertex $u$ and to a half-deficient vertex $w$.
By $(C2)$, $d\geq 4$. Denote by $u'$ the neighbour of degree $2$ of $u$. Denote by $w'$ the neighbour of degree greater than $2$ of $w$ distinct from $v$. Note that by $(C2)$ and $(C4)$, neither $u$ nor $u'$ is adjacent with $w$. If $w$ is a bad $3$-vertex, let $w''$ be its neighbour of degree $2$. Colour the graph $H'=H-\{vu,uu',vw\}$ by minimality. In the case when $w$ is a bad $3$-vertex we uncolour the edge $ww''$. Regardless if $d(w)=2$ or $d(w)=3$, for $vw$ there are at most $(d-2)+1$ forbidden colours of the edges adjacent with it and $1$ more constraint to guarantee $s(w)\neq s(w'')$ (if $d(w)=3$) or $s(w)\neq s(w')$ (if $d(w)=2$). Analogously, as the colour of $uu'$ is not yet fixed, there are $d-2$ colours of the edges incident with $v$ forbidden for $uv$ and at most two more so that $s(u)\neq s(u')$ and $s(v)\neq s(w)$. Therefore we have at least $(k+1)-d\geq \frac{k}{3}+1$ colours available for both, $vu$ and $vw$, thus by Lemma~\ref{MartheLemma} we may extend our proper colouring on these two edges obtaining at least $2(\frac{k}{3}+1)-3=\frac{2}{3}k-1\geq (d-2)+1$ distinct sums for $v$, one of which is different from the sums of all neighbours of $v$ other than $u$ and $w$. Then we easily complete the construction of an nsd $(k+1)$-colouring of $H$ choosing a right colour for $uu'$ and one for $ww''$ (if $d(w)=3$) as in $(C7)$. Thus we obtain an extension of the colouring to the whole $H$, a contradiction.\\


\item  Assume there is a vertex $v$ of degree $d$ adjacent to exactly one bad $2$-vertex $u$, at least one half-deficient vertex and to at most $k-d$ vertices which are neither deficient nor half-deficient. 
Denote by $u'$ the neighbour of $u$ of degree $2$. Colour $H'= H-\{vu,uu'\}$ by minimality. Then, first we choose a colour for $vu$ so that the colouring is proper ($d-1$ constraints), $s(u)\neq s(u')$ ($1$ constraint) and the sum of $v$ is distinct from the sum of every its neighbour which is neither deficient nor half-deficient (there are at most $k-d$ of these). This is feasible, as we have altogether at most $(d-1)+1+(k-d)=k$ constraints. Subsequently, we choose an appropriate colour for $uu'$ (avoiding at most $4$ constraints). Recall now that $v$ is adjacent to at least one half-deficient vertex. By $(C9)$, $d\geq\frac{2k+1}{3}$, and thus $v$ is sum-distinguished from all its $2$-neighbours by (\ref{distinguished2vertices_ineq}). Hence, $v$ can only be in conflict with its adjacent half-deficient vertices which are bad $3$-vertices. For every such vertex we can however similarly as above adjust the colour on the edge joining it with the vertex of degree $2$ in order to eliminate this potential conflict. Finally we obtain an extension of the colouring to the whole $H$, a contradiction.\\

\item Suppose there is a vertex $v$ of degree $d$ adjacent to exactly one $1$-vertex $u$
and to at most $k-d+1$ vertices which are neither deficient nor half-deficient.
Colour the graph $H'=H-u$ by minimality. Then, we choose a colour for $vu$ so that the colouring is proper ($d-1$ constraints) and the sum of $v$ is distinct from the sum of every its neighbour which is neither deficient nor half-deficient (there are at most $k-d+1$ of these). This is feasible, as we have altogether at most $(d-1)+(k-d+1)=k$ constraints. Since by $(C1)$, $d\geq\frac{k+3}{2}$, then $v$ is sum-distinguished from all its $2$-neighbours by (\ref{distinguished2vertices_ineq}), and hence can only be in conflict with its adjacent bad $3$-vertices. For every such vertex we can however similarly as above adjust the colour on the edge joining it with the vertex of degree $2$ in order to eliminate this potential conflict. Thus we obtain an extension of the colouring to the whole $H$, a contradiction.\\

\item Assume there is a $5$-vertex $v$ adjacent to $5$ half-deficient vertices $u_1,\ldots,u_5$. Colour the graph $H'=H-\{vu_1,vu_2\}$ by minimality. Without loss of generality we may assume that $u_1,u_2$ are not adjacent in $H$. In the obtained colouring, for every $u_i$ which is a bad $3$-vertex we uncolour an edge joining it with a vertex of degree $2$, $i=1,\ldots,5$. If $k\leq 8$, we then first colour $vu_2$ properly ($4$ constraints) so that $u_2$ is sum-distinguished ($1$ constraint) from its neighbour other than $v$
    which, if possible (\emph{i.e.} in the case when $d(u_2)=3$) is a $2$-vertex.
    Then we colour $vu_1$ properly ($5$ constraints) so that $u_1$ is sum-distinguished ($1$ constraint) from its neighbour other than $v$
    which, if possible (\emph{i.e.} in the case when $d(u_1)=3$) is a $2$-vertex.
    As $k\leq 8$, by (\ref{distinguished2vertices_ineq}), $v$ is sum-distinguished from all its neighbours of degree $2$. In order to distinguish it from bad $3$-neighbours, we subsequently choose new colours for the formerly uncoloured edges incident with them, which is possible as $k+1>6$. If on the other hand $k\geq 9$, we have at most $4$ colours blocked for each of $vu_1$ and $vu_2$ by the colours of their respective adjacent edges and further $2$ for each $vu_i$, $i=1,2$, to avoid $s(v)=s(u_{3-i})$, resp., and the same sum at $u_i$ and its neighbour of the least degree other than $v$. 
    Thus both edges $vu_1$ and $vu_2$ have at least $k+1-4-2\geq 4$ colours available left, hence by Lemma~\ref{MartheLemma}, we may properly extend the colouring to $vu_1$ and $vu_2$ obtaining at least $2\times 4-3 = 5$ different sums at $v$. We choose one of these extensions so that $s(v)\neq s(u_3),s(v)\neq s(u_4),s(v)\neq s(u_5)$. At the end, if necessary, we analogously as in the previous case adjust the colours of uncoloured edges incident with bad $3$-vertices adjacent with $v$. Thus we obtain an extension of the colouring to the whole $H$, a contradiction.\\

\item 	Suppose there exists a $4$-vertex $v$ adjacent to at least $3$ half-deficient vertices $u_1,u_2,u_3$.
If $u_1$ is adjacent to $u_2$, by $(C3)$ and $(C4)$ it means that one of these vertices, say $u_2$ is a bad $3$-vertex, and the other ($u_1$) is a good $2$-vertex. Then we colour $H'=H-\{u_1v,u_1u_2\}$ by minimality. Next we extend this proper colouring to $vu_1$ (at most $3$ forbidden colours of the adjacent edges) so that $s(u_1)\neq s(u_2)$ and $v$ is sum-distinguished from its remaining two neighbours (other than $u_1$ and $u_2$). Finally we choose a colour for $u_1u_2$ avoiding the colours of its three adjacent edges and creating no sum conflicts (additional at most $3$ constraints), a contradiction. By symmetry, we may thus assume that $u_1,u_2,u_3$ form an independent set in $H$, and denote by $w$ the remaining neighbour of $v$.
    \begin{itemize}
    \item
    If $k=6$, we colour $H-\{vu_1,vu_2,vu_3\}$ by minimality, and for every $u_i$ which is a bad $3$-vertex we uncolour an edge joining it with a vertex of degree $2$, for $i=1,2,3$. For every $vu_i$ we then have at most $2$ colours blocked by the colours of its adjacent edges and further at most one to avoid a sum-conflict between $u_i$ ($i=1,2,3$) and its neighbour of the least degree other than $v$. We thus have at least $(k+1)-2-1=4$ available colours left for every $vu_i$ ($i=1,2,3$). Therefore, we may first choose a colour for $vu_1$ so that $\max\{c(vu_1),c(vw)\}\geq 5$ (note that this guarantees that if $u_i$, $i\in\{2,3\}$, is of degree $2$, then $s(v)\neq s(u_i)$). Thus we are left with lists of size at least $3$ of available colours for $vu_2$ and $vu_3$, from which it is sufficient to choose distinct colours so that $s(v)\neq s(u_1)$ and $s(v)\neq s(w)$. This is feasible by Lemma~\ref{MartheLemma} because we may obtain at least $3$ different sums at $v$. By our construction, in order to eliminate the remaining potential conflicts it is then sufficient to choose appropriate colours for the formerly uncoloured edges incident with bad $3$-vertices.
    \item
    We may thus assume that $k\geq 7$. Then we colour $H-\{vu_1,vu_2\}$ by minimality, and for every $u_i$, $i=1,2$ which is a bad $3$-vertex we uncolour an edge joining it with a vertex of degree $2$. Then for every $vu_i$, $i=1,2$, we have forbidden at most $3$ colours of its adjacent edges and at most $2$ more constraints guaranteeing $s(v)\neq s(u_{3-i})$ ($i=1,2$) and $s(u_i)\neq s(u'_i)$, where $u'_i$ ($i=1,2$) is the neighbour of $u_i$ of minimal degree distinct from $v$. Altogether we are left with lists of available colours of sizes at least $k+1-3-2\geq 3$, and need only choose distinct values from these two lists so that $s(v)\neq s(w)$ and $s(v)\neq s(u_3)$. This is feasible by Lemma~\ref{MartheLemma} because we may obtain at least $3$ different sums at $v$. Again by our construction, in order to eliminate the remaining potential conflicts it is then sufficient to choose appropriate colours for the formerly uncoloured edges incident with bad $3$-vertices.
    \end{itemize}

In each case, we obtain an extension of the colouring to the whole $H$, a contradiction. $\blacksquare$
     %

\end{enumerate}

\end{proof}

\subsection{Discharging procedure}

In this subsection we use the discharging technique exploiting the vertices of the graph $H$.
For this aim we first define the \emph{weight
function} $\omega: V(H) \rightarrow \mathbb{R}$ by setting $\omega(x)=d(x)-3$ for every $x\in V(H)$.
%
Next we shall apply so called \emph{Ghost vertices method}, introduced earlier by Bonamy, Bousquet and Hocquard~\cite{BonamyEtAl},
and based on the following observation (where given any subsets $U,U'\subseteq V(H)$ and a vertex $v$, $d_U(v)$ denotes the number of neighbours of $v$ from $U$, while $E(U,U')$ is the set of edges joining $U$ and $U'$ in the graph $H$).

\begin{observation}\label{obs1v}
\noindent Let $V_1 \cup V_2$ be a partition of $V(H)$ where, say $V_1$ is the set of vertices of degree at least $2$ and $V_2$ the set of vertices of degree $1$ in $H$;
\begin{itemize}
\item every vertex $u$ in $H$ has an initial weight $w(u)=d(u)-3$.
\item If we can discharge the weights in $H$ so that:
\begin{enumerate}
\item every vertex in $V_1$ has a non-negative weight;
\item and every vertex $u$ in $V_2$ has a final weight of at least $d(u)-3+d_{V_1}(u)$, then\\

for $\omega'$ the new weight assignment, we have $\sum_{v \in V_2} (d(v)-3+d_{V_1}(v)) \leq \sum_{v \in V_2} \omega'(v)$, as well as\\

$\sum_{v \in V} \omega(v)=\sum_{v \in V} \omega'(v)$ and $\sum_{v \in V_1} \omega'(v) \geq 0$. Therefore,

\begin{eqnarray}
   \sum_{v \in V_1} (d_{V_1}(v)-3) & \geq & \sum_{v \in V_1} (d_{V_1}(v)-3) + \sum_{v \in V_2} (d(v)-3+d_{V_1}(v)) - \sum_{v \in V_2} \omega'(v) \nonumber\\
    & \geq & \sum_{v \in V_1} (d_{V_1}(v)-3) + |E(V_1,V_2)|+\sum_{v \in V_2} (d(v)-3) - \sum_{v \in V_2} \omega'(v) \nonumber\\
     & \geq & \sum_{v \in V_1} (d(v)-3) + \sum_{v \in V_2} (d(v)-3) - \sum_{v \in V_2} \omega'(v) \nonumber\\
     & \geq & \sum_{v \in V} \omega(v) - \sum_{v \in V_2} \omega'(v) \nonumber\\
   & \geq & \sum_{v \in V_1} \omega'(v) \nonumber\\
    & \geq & 0. \nonumber
\end{eqnarray}
\end{enumerate}

Thus we can conclude that ${\rm mad}(H) \geq {\rm mad}(H[V_1]) \geq 3$.

\end{itemize}

\end{observation}

In other words, the vertices in $V_2$ can be seen but, in a way, do not contribute to the sum analysis.
\bigskip

In order to finish the proof of Theorem~\ref{th:summad}, it suffices to obtain a contradiction, \emph{e.g.} with the fact that ${\rm mad}(H)<3$,
implying that in fact no counterexample to its thesis may exist.
By Observation~\ref{obs1v}, it is thus enough to
%
redistribute the weight (defined by $\omega$ above) in $H$ 
so that every vertex of degree at least $2$ has a non-negative resulting weight and every vertex of degree one has weight at least $-1$.

\medskip

The discharging rules we shall use for this aim are defined as follows:

\begin{enumerate}
\item[(R1)] A vertex of degree $d\geq 5$ gives $1$ to every adjacent $1$-vertex.
\item[(R2)] A vertex of degree $d\geq 4$ gives $1$ to every adjacent bad $2$-vertex.
\item[(R3)] A vertex of degree $d\geq 3$ 
gives $\frac{1}{2}$ to every adjacent good $2$-vertex.
\item[(R4)] A vertex of degree $d\geq 4$ gives $\frac{1}{2}$ to every adjacent bad $3$-vertex.
\end{enumerate}


Let $v$ be a vertex in $H$. We consider different cases depending on the degree of $v$.

\begin{itemize}
\item Assume $d(v)=1$. By $(C1)$, $v$ is adjacent to a vertex of degree at least $5$. Thus, by $(R1)$, $v$ receives $1$. So every vertex of degree $1$ in $H$ has an initial weight of $-2$, gives nothing according to our rules and receives $1$, hence has the final weight of $-1$.

\item Assume $d(v)=2$. First, suppose that $v$ is a bad $2$-vertex. Then by $(C2)$, $v$ is adjacent to a vertex of degree at least $4$, and thus receives at least $1$ by $(R2)$ (and gives away nothing according to the rules above). Suppose now that $v$ is a good $2$-vertex. Then by $(C1)$, $v$ is adjacent with two vertices of degree at least $3$, and thus receives at least $\frac{1}{2}$ from both by $(R3)$ (and gives away nothing). In both cases $v$ has a non-negative final weight.

\item Assume $d(v)=3$. First, suppose that $v$ is a bad $3$-vertex. Then by $(C3)$ it gives away (at most) $\frac{1}{2}$ due to rule $(R3)$, but also receives at least $2\times\frac{1}{2}$ by $(R4)$, as $(C3)$ implies that $v$ must have two neighbours of degree at least $4$. Suppose now that $v$ is a good $3$-vertex. Then $v$ gives away nothing and receives nothing. In both cases $v$ has a non-negative final weight.

\item Assume $d(v) \ge 4$. Then consider the following subcases:
\begin{itemize}
\item if $v$ has at least $2$ deficient neighbours, then by
$(C5)$ and $(C6)$, these are both of degree $1$ and by $(C1)$, $d(v) \ge 5$. Moreover, in such a case, additionally by $(C7)$, $v$ is adjacent with no other deficient or half-deficient vertices (except $1$-vertices), while by $(C8)$, $v$ can be adjacent with at most $d-3$ vertices of degree $1$, and thus by $(R1)$, $\omega'(v)\geq 0$;
\item if $v$ has exactly $1$ deficient neighbour, then we may assume that it has at least one half-deficient neighbour, as otherwise by $(R1)$ or $(R2)$, $\omega'(v)\geq (d(v)-3)-1\geq 0$. Thus by $(C9)$ and $(C1)$, $d(v)\geq 5$. On the other hand, by $(C10)$ and $(C11)$, at least $k-d(v)+1$ neighbours of $v$ are neither deficient nor half-deficient, and thus by $(R1)$, $(R2)$, $(R3)$ and $(R4)$, $\omega'(v)\geq (d(v)-3)-1-\frac{1}{2}[(d(v)-1)-(k-d(v)+1)]=\frac{1}{2}k-3\geq 0$;
\item assume then finally that $v$ has no deficient neighbours. If $d(v) \ge 6$, then by $(R3)$ and $(R4)$, $\omega'(v)\geq d(v)-3-\frac{1}{2}d(v)\geq 0$. Consider now the case where $d(v) \leq 5$:
\begin{itemize}
\item if $d(v)=5$, then by $(C12)$, $v$ has at most $4$ half-deficient neighbours, and thus by $(R3)$ and $(R4)$, $\omega'(v)\geq 2-4\times\frac{1}{2}\geq 0$;
\item if $d(v)=4$, then by $(C13)$, $v$ has at most $2$ half-deficient neighbours, and thus by $(R3)$ and $(R4)$, $\omega'(v)\geq 1-2\times\frac{1}{2}\geq 0$.
\end{itemize}
In both cases $v$ has a non-negative final weight.
\end{itemize}
\end{itemize}

This, by Observation~\ref{obs1v} completes the proof of Theorem~\ref{th:summad}. $\blacksquare$

\end{document}